\documentclass[10pt]{amsart}

\usepackage{geometry,graphicx,amssymb,amsmath,amsbsy,eucal,amsfonts,mathrsfs,amscd,bm,tcolorbox}
\usepackage{float}
\usepackage{tikz}
\usetikzlibrary{patterns}
\usepackage{subcaption}
\numberwithin{equation}{section}

\allowdisplaybreaks[3]

\newtheorem{theorem}{Theorem}[section]
\newtheorem{lemma}[theorem]{Lemma}

\theoremstyle{definition}

\theoremstyle{remark}
\newtheorem{remark}[theorem]{Remark}
\newtheorem{example}[theorem]{Example}

\newcommand{\bl}{\bigl\langle}
\newcommand{\br}{\bigr\rangle}

\newcommand{\bld}[1]{\boldsymbol{#1}}

\newcommand{\bn}{\bld{n}}

\newcommand{\uu}{\mathsf{u}}
\newcommand{\ww}{\mathsf{w}}

\newcommand{\lla}{\mathsf{l}}

\newcommand{\Lla}{\Lambda}

\newcommand{\LL}{\mathsf{L}}

\newcommand{\pdt}{\partial_{\Delta t}}
\newcommand{\pt}{\partial_{t}}
\newcommand{\dt}{\Delta t}

\def\fU{\mathfrak{U}}

\def\fW{\mathfrak{W}}
\def\fL{\mathfrak{L}}

\def\cG{\mathcal{G}}
\def\cH{\mathcal{H}}
\def\cR{\mathcal{R}}

\title[Parabolic-Parabolic Interface Problems]{An Improved Robin-Robin Coupling Method for Parabolic-Parabolic Interface Problems}

\author{Erik Burman}
\address{$^1$Department of Mathematics, University College London, London, UK–WC1E 6BT, United Kingdom}
\email{e.burman@ucl.ac.uk}
\author{Miguel A. Fern\'andez}
\address{$^2$Sorbonne Universit\'e \& CNRS, UMR 7598 LJLL, 75005 Paris France -- Inria Paris, 75012 Paris, France}
\email{miguel.fernandez@inria.fr}
\author{Johnny Guzm\'an}
\address{$^3$Division of Applied Mathematics, Brown University, 182 George Street, Providence, RI 02912, USA}
\email{johnny\_guzman@brown.edu}
\author{Sijing Liu}
\address{$^4$Department of Mathematical Sciences, Worcester Polytechnic Institute, 100 Institute Road, Worcester, MA 01609, USA}
\email{sliu13@wpi.edu}

\begin{document}

\begin{abstract}
  We consider a loosely coupled, non-iterative Robin-Robin coupling method proposed and analyzed in {\em [Numer. Algorithms, 99:921-948, 2025]} for a parabolic-parabolic interface problem. We modify the first step of the scheme so that several error difference quantities remain higher order convergence without requiring additional assumptions. Numerical results are presented to support our findings.
\end{abstract}

\maketitle

\section{Introduction}

For the Robin-Robin coupling method considered in \cite{burman2024estimates}, we observed in \cite{erratum} that the following quantities
\begin{equation*}
Z^N(\pdt W, \pdt U, \pdt \Lambda) + \sum_{n=1}^{N-1} S^{n+1}(\pdt W, \pdt U, \pdt \Lambda)
\end{equation*}
and
\begin{equation*}
Z^N(\pdt^2 W, \pdt^2 U, \pdt^2 \Lambda) + \sum_{n=2}^{N-1} S^{n+1}(\pdt^2 W, \pdt^2 U, \pdt^2 \Lambda)
\end{equation*}
do not converge as $O((\dt)^2)$ as expected, where $Z$ and $S$ are defined as follows:
\begin{alignat*}{1}
Z^{n+1}(\psi,\phi,\theta):= & \frac{1}{2} \|\phi^{n+1}\|_{L^2(\Omega_f)}^2+\frac{1}{2}\|\psi^{n+1}\|_{L^2(\Omega_s)}^2 +\frac{\dt \alpha}{2} \|\phi^{n+1}\|_{L^2(\Sigma)}^2+ \frac{\dt}{2\alpha} \|\theta^{n+1}\|_{L^2(\Sigma)}^2, \\ 
S^{n+1}(\psi, \phi,\theta):= & \dt\Big(\nu_f  \|\nabla \phi^{n+1}\|_{L^2(\Omega_f)}^2+  \nu_s \|\nabla \psi^{n+1}\|_{L^2(\Omega_s)}^2\Big) +  \frac{1}{2} \Big(\|\psi^{n+1}-\psi^n\|_{L^2(\Omega_s)}^2 \\
&+ \|\phi^{n+1}-\phi^n\|_{L^2(\Omega_f)}^2\Big)+ \frac{\alpha \dt}{2}\|\phi^{n+1}-\phi^n+  \frac{1}{\alpha} ( \theta^{n+1}-\theta^n)\|_{L^2(\Sigma)}^2.
\end{alignat*}
In particular, we observed that the sum terms in $\sum S^{n+1}$ do not converge as $O((\dt)^2)$, even though we do observe $O((\dt)^2)$ convergence for the terms in $Z^N$ numerically in \cite{burman2024estimates}. However, the $O((\dt)^2)$ convergence for the above expressions is a crucial ingredient for the defect correction methods proposed and analyzed in \cite{ppcorrection}. We examined this issue in \cite{erratum} and pointed out that, under suitable assumptions on the exact solution of the interface problem, we can obtain $O((\dt)^2)$ convergence for both the terms in $Z^N$ and in $\sum S^{n+1}$. Specifically, the additional assumptions state that some of the time derivatives of the exact solutions must vanish during the initial steps. 

In this note, we present a modified scheme that differs from that in \cite{burman2023loosely,burman2024estimates} only at the first time step, but improves the convergence rates of $\sum S^{n+1}$ to $O((\dt)^2)$ without requiring any additional assumptions on the exact solutions of the interface problem. However, the first three time steps of the modified scheme are coupled, so it requires solving a large system to obtain the numerical solutions at these time levels. We also briefly discuss the analysis of this modified scheme. The main idea is to modify the first time step so that the initial terms affecting the convergence rates are canceled during the analysis. We follow the notation used in \cite{burman2024estimates} throughout this note.

\section{The Original Robin-Robin Coupling Method}
We consider the interface problem
\begin{subequations}\label{eq:ppinterface}
\begin{alignat}{2}
\pt \uu-\nu_f \Delta \uu=&0, \quad && \text{ in } [0, T] \times \Omega_f, \nonumber\\
\uu(0, x)=& \uu_0(x), \quad  && \text{ in } \Omega_f, \\
\uu=&0,  \quad && \text{ on }  [0, T] \times \Gamma^f_D,\nonumber\\
\frac{\partial\uu}{\partial n}=&0,  \quad && \text{ on }  [0, T] \times \Gamma^f_{Ne},\nonumber\\
\nonumber\\
\pt \ww-\nu_s \Delta \ww=&0, \quad && \text{ in } [0, T] \times \Omega_s, \nonumber\\
\ww(0, x)=& \ww_0(x), \quad  && \text{ in } \Omega_s, \\
\ww=&0,  \quad && \text{ on }  [0, T] \times \Gamma^s_D,\nonumber\\
\frac{\partial\ww}{\partial n}=&0,  \quad && \text{ on }  [0, T] \times \Gamma^s_{Ne},\nonumber\\
\nonumber\\
\ww - \uu =& 0, \quad && \text{ on } [0,T] \times \Sigma, \\
\nu_s \nabla \ww \cdot \bn_s + \nu_f \nabla \uu \cdot \bn_f =& 0, \quad && \text{ on } [0,T] \times \Sigma,
\end{alignat}  
\end{subequations}
where $\bn_f$ and $\bn_s$ are the outward facing normal vectors for $\Omega_f$ and $\Omega_s$, respectively. We assume that the initial data is smooth and that $\uu$ and $\ww$ is smooth on $\Omega_f$ and $\Omega_s$ respectively.   

The Robin–Robin splitting method is given as in \cite{burman2024estimates}. Find $w^{n+1}  \in V_s$, $u^{n+1} \in V_f$, and $\lambda^{n+1} \in V_g $ such that, for $n\geq 0$,
\begin{subequations}\label{eq:orr}
\begin{alignat}{2}
(\pdt w^{n+1}, z)_s+\nu_s(\nabla w^{n+1}, \nabla z)_s + \alpha \bl w^{n+1} - u^n, z \br + \bl \lambda^n, z\br=&0 , \quad && z \in V_s, \\
(\pdt u^{n+1}, v)_f+\nu_f(\nabla u^{n+1}, \nabla v)_f- \bl \lambda^{n+1}, v\br=&0, \quad &&v \in V_f,  \\
\bl \alpha(u^{n+1}-w^{n+1})+(\lambda^{n+1}-\lambda^n), \mu \br=&0, \quad && \mu \in V_g. 
\end{alignat}
\end{subequations}

\section{The Improved Method}

The improved scheme is defined as follows. For $n=0$, find $w^{1} \in V_s$, $u^{1} \in V_f$, and $\lambda^{1} \in V_g $ such that 
\begin{subequations}\label{eq:firststep}
\begin{alignat}{2}
(\pdt w^{2}, z)_s+\nu_s(\nabla w^{1}, \nabla z)_s + \alpha \bl w^{1} +u^2-2u^1, z \br + \bl 2\lambda^1-\lambda^2, z\br\\
=(\pdt\ww^2-\pdt\ww^1,z)_s+\alpha\dt\bl G_1^2, z\br-\dt\bl G_2^2,z\br, &\quad && z \in V_s, \nonumber\\
(\pdt u^{2}, v)_f+\nu_f(\nabla u^{1}, \nabla v)_f+ \bl\alpha(\delta u^3-\delta u^2)+\lambda^{3}-2\lambda^2, v\br\\
=(\pdt\uu^2-\pdt\uu^1,v)_f+\alpha\dt\bl G_1^3, v\br+\dt\bl G_2^3,v\br, &\quad &&v \in V_f, \nonumber \\
\bl \alpha(-u^{3}+2u^2-w^{1}), \mu \br+\bl \lambda^2-\lambda^1, \mu \br\\
=-\alpha\dt\bl G_1^3,\mu\br+\dt\bl G_2^2,\mu\br&, \quad && \mu \in V_g. \nonumber
\end{alignat}
\end{subequations}
and for $n\ge1$, find $w^{n+1}  \in V_s$, $u^{n+1} \in V_f$, and $\lambda^{n+1} \in V_g $ such that,
\begin{subequations}\label{eq:nextsteps}
\begin{alignat}{2}
(\pdt w^{n+1}, z)_s+\nu_s(\nabla w^{n+1}, \nabla z)_s + \alpha \bl w^{n+1} - u^n, z \br + \bl \lambda^n, z\br=&0 , \quad && z \in V_s, \label{s1}\\
(\pdt u^{n+1}, v)_f+\nu_f(\nabla u^{n+1}, \nabla v)_f- \bl \lambda^{n+1}, v\br=&0, \quad &&v \in V_f,  \label{s2}\\
\bl \alpha(u^{n+1}-w^{n+1})+(\lambda^{n+1}-\lambda^n), \mu \br=&0, \quad && \mu \in V_g.  \label{s3}
\end{alignat}
\end{subequations}

Notice that this modified scheme differs from the original scheme only at the very first step when $n=0$. However, the first step \eqref{eq:firststep} requires information about the numerical solutions at $t_2$ and $t_3$, as well as the exact solutions at $t_1$, $t_2$, and $t_3$. This implies that the first three steps of the method are coupled, hence, they should be solved as a single large system. We also believe that the exact solutions on the right-hand side of the method can be replaced by high-order numerical approximations. This could potentially be achieved using a higher-order monolithic method for the first few steps.

We now briefly discuss the analysis of the modified scheme \eqref{eq:firststep}-\eqref{eq:nextsteps}. For $n=0$, we write the exact solution as follows:

\begin{alignat*}{2}
(\pdt \ww^2, z)_s+\nu_s(\nabla \ww^1, \nabla z)_s + \alpha \bl \ww^1 + \uu^2- 2\uu^1, z \br + \bl 2\lla^1-\lla^2, z\br=& -(h_3^{2},z)_s+\alpha\bl g_1^2,z\br - \bl g_2^2,z\br, \quad && z \in V_s, \\
(\pdt \uu^2, v)_f+\nu_f(\nabla \uu^1, \nabla v)_f+ \bl \alpha(\delta\uu^3-\delta\uu^2)+\lla^3-2\lla^2, v\br=&\alpha\dt\bl G_1^3, v\br+\dt\bl G_2^3,v\br-(h_4^2,v)_f, \quad &&v \in V_f,  \\
\alpha\bl -\uu^3+2\uu^2-\ww^1,\mu\br+\bl\lla^2-\lla^1,\mu\br=&-\alpha\dt\bl G_1^3,\mu\br+\bl g_2^2,\mu\br, \quad && \mu \in V_g.  
\end{alignat*}

For $n\ge1$, we write the exact solution as follows:
\begin{subequations}\label{s}
\begin{alignat}{2}
(\pdt \ww^{n+1}, z)_s+\nu_s(\nabla \ww^{n+1}, \nabla z)_s + \alpha \bl \ww^{n+1} - \uu^n, z \br + \bl \lla^n, z\br=& L^{n+1}_1(z) - L^{n+1}_3(z)  , \quad && z \in V_s, \label{s1}\\
(\pdt \uu^{n+1}, v)_f+\nu_f(\nabla \uu^{n+1}, \nabla v)_f- \bl \lla^{n+1}, v\br=&L^{n+1}_2(v), \quad &&v \in V_f,  \label{s2}\\
\bl \alpha(\uu^{n+1}-\ww^{n+1})+(\lla^{n+1}-\lla^n), \mu \br=&L^{n+1}_3(\mu), \quad && \mu \in V_g,  \label{s3}
\end{alignat}
\end{subequations}
where 
\begin{alignat*}{1}
L^{n+1}_1(z) :=& \alpha \bl g_1^{n+1}, z \br - ( h_1^{n+1}, z)_s,\\
L^{n+1}_2(v):=&-(h_2^{n+1},v)_f,\\
L^{n+1}_3(\mu):=&  \bl g_2^{n+1}, \mu \br,
\end{alignat*}
and
\begin{alignat*}{2}
h_1^{n+1} &:= \pt \ww^{n+1} - \pdt \ww^{n+1}, \quad && g_1^{n+1} := \uu^{n+1}-\uu^n, \\
h_2^{n+1} &:= \pt \uu^{n+1} - \pdt \uu^{n+1}, \quad && g_2^{n+1} := \lla^{n+1} - \lla^n,\\
h_3^{n+1} &:= \pt \ww^{n} - \pdt \ww^{n+1}, \quad && h_4^{n+1} := \pt \uu^{n} - \pdt \uu^{n+1},\\
G_1^{n+1} &:= \pdt g_1^{n+1},\quad && G_2^{n+1}:= \pdt g_2^{n+1}.
\end{alignat*}
\subsection{Error Equations}

We use the notations  
\begin{alignat*}{1}
U^n=& \uu^n-u^n, \quad W^n=\ww^n-w^n, \quad \Lambda^{n}=\lla^n-\lambda^n \\
\mathcal{U}^{n+1}=& \pdt U^{n+1}, \quad \mathcal{W}^{n+1}=\pdt W^{n+1}, \quad \LL^{n+1}=\pdt \Lambda^{n+1}\\
\mathfrak{U}^{n+1}=& \pdt \mathcal{U}^{n+1}, \quad \mathfrak{W}^{n+1}=\pdt \mathcal{W}^{n+1}, \quad \mathfrak{L}^{n+1}=\pdt \LL^{n+1},
\end{alignat*}
and 
\begin{alignat*}{2}
H_j^{n+1}&=\pdt h_j^{n+1},\\
\cG_j^{n+1}&=\pdt G_j^{n+1},  \\
\cH_j^{n+1}&=\pdt H_j^{n+1},
\end{alignat*}
for $j=1,2$.

By subtracting the exact solutions and the numerical solutions, we have the following error equations. For $n=0$, we have
\begin{subequations}\label{eq:d11}
\begin{align}
(\pdt W^2,z)_s+\nu_s(\nabla W^{1}, \nabla z)_s +\alpha \bl W^{1} +U^2-2U^1, z \br& + \bl 2\Lla^1-\Lla^2, z\br \\
=& -(h_1^{1},z)_s+\alpha\bl g_1^1,z\br - \bl g_2^1,z\br, \nonumber\\
(\pdt U^2,z)_f+\nu_f (\nabla U^{1}, \nabla v)_f+ \bl\alpha(\delta U^3-\delta U^2)+\Lla^{3}-2\Lla^2, v\br &= -(h_2^1,v)_f, \\
\alpha\bl -U^{3}+2U^2-W^{1}, \mu \br+\bl \Lla^2-\Lla^1, \mu \br &= \bl g_2^1,\mu\br
\end{align}
\end{subequations}
and for $n\ge 1$:
\begin{subequations}\label{eq:dstandard1}
\begin{alignat}{1}
(\pdt W^{n+1}, z)_s+ \nu_s(\nabla W^{n+1}, \nabla z)_s+\alpha \bl (W^{n+1} - U^n), z\br+ \bl \Lla^n, z \br =& L^{n+1}_1(z) - L^{n+1}_3(z) , \\
 (\pdt U^{n+1}, v)_f+ \nu_f (\nabla U^{n+1}, \nabla v)_f -  \bl  \Lla^{n+1},v\br
= & L^{n+1}_2(v), \\
 \alpha \bl U^{n+1} -  W^{n+1},\mu\br +\bl \Lla^{n+1} - \Lla^{n}, \mu\br = & L^{n+1}_3(\mu).
\end{alignat}
\end{subequations}

\subsubsection{First-order time difference}
Furthermore, using \eqref{eq:d11} and \eqref{eq:dstandard1}, we have, for $n=1$
\begin{subequations}\label{eq:dd11}
\begin{alignat}{1}
\nu_s(\nabla \mathcal{W}^{2}, \nabla z)_s +\alpha \bl \mathcal{W}^2 - \mathcal{U}^2, z\br+\bl \LL^2, z\br =& R^{2}_1(z) - R^{2}_3(z), \label{eq:dd111}\\
\nu_f (\nabla \mathcal{U}^{2}, \nabla v)_f-\alpha\bl \mathcal{U}^3-\mathcal{U}^2, v\br-\bl\LL^3,v\br 
= & R^{2}_2(v), \label{eq:dd112}\\
\alpha\bl \mathcal{U}^3-\mathcal{W}^2,\mu\br = & R^{2}_3(\mu),\label{eq:dd113}
\end{alignat}
\end{subequations}
and $n\ge2$:
\begin{subequations}\label{eq:ddstandard1}
\begin{alignat}{1}
(\pdt \mathcal{W}^{n+1}, z)_s+ \nu_s(\nabla \mathcal{W}^{n+1}, \nabla z)_s+\alpha \bl (\mathcal{W}^{n+1} - \mathcal{U}^n), z\br+ \bl \LL^n, z \br =& R^{n+1}_1(z) - R^{n+1}_3(z) , \label{eq:ddstandard11}\\
 (\pdt \mathcal{U}^{n+1}, v)_f+ \nu_f (\nabla \mathcal{U}^{n+1}, \nabla v)_f -  \bl  \LL^{n+1},v\br
= & R^{n+1}_2(v),\label{eq:ddstandard21} \\
 \alpha \bl \mathcal{U}^{n+1} -  \mathcal{W}^{n+1},\mu\br +\bl \LL^{n+1} - \LL^{n}, \mu\br = & R^{n+1}_3(\mu),\label{eq:ddstandard31}
\end{alignat}
\end{subequations}
where
\begin{alignat*}{1}
R_1(z) :=& \alpha \bl G_1^{n+1}, z \br - ( H_1^{n+1}, z)_s,\\
R_3(v):=&-(H_2^{n+1},v)_f,\\
R_4(\mu):=&  \bl G_2^{n+1}, \mu \br.
\end{alignat*}

\subsubsection{Second-order time difference}
Subtract \eqref{eq:dd11} and \eqref{eq:ddstandard1} and divide them by $\dt$, we have, for $n=2$,
\begin{subequations}\label{eq:ddd1}
\begin{alignat}{1}
\frac{1}{\dt}(\mathfrak{W}^3, z)_s+\nu_s(\nabla \mathfrak{W}^{3}, \nabla z)_s +\alpha \bl \mathfrak{W}^3, z\br =& \cR^{3}_1(z) - \cR^{3}_3(z), \label{eq:ddd11}\\
\frac{1}{\dt}(\mathfrak{U}^3, v)_f+\nu_f (\nabla \mathfrak{U}^{3}, \nabla v)_f+\alpha\bl \mathfrak{U}^3, v\br 
= & \cR^{3}_2(v), \label{eq:ddd12}\\
\bl \mathfrak{L}^3-\alpha\mathfrak{W}^3,\mu\br = & \cR^{3}_3(\mu),\label{eq:ddd13}
\end{alignat}
\end{subequations}
and $n\ge3$,
\begin{subequations}\label{errorRRd}
\begin{alignat}{1}
(\pdt \fW^{n+1}, z)_s+ \nu_s(\nabla \fW^{n+1}, \nabla z)_s+\alpha \bl (\fW^{n+1} - \fU^n), z\br+ \bl \fL^n, z \br =& \cR^{n+1}_1(z) - \cR^{n+1}_3(z) ,\label{errorRRd1d} \\
 (\pdt \fU^{n+1}, v)_f+ \nu_f (\nabla \fU^{n+1}, \nabla v)_f -  \bl  \fL^{n+1},v\br
= & \cR^{n+1}_2(v), \label{errorRRd2d}\\
 \alpha \bl \fU^{n+1} -  \fW^{n+1},\mu\br +\bl \fL^{n+1} - \fL^{n}, \mu\br = & \cR^{n+1}_3(\mu),\label{errorRRd3d}
\end{alignat}
\end{subequations}
where 
\begin{alignat*}{1}
\cR^{n+1}_1(z) :=& \alpha \bl \cG_1^{n+1}, z \br - ( \cH_1^{n+1}, z)_s,\\
\cR^{n+1}_2(v):=&-(\cH_2^{n+1},v)_f,\\
\cR^{n+1}_3(\mu):=&  \bl \cG_2^{n+1}, \mu \br.
\end{alignat*}

\subsection{Sketch of the Analysis}

Take $z=\mathfrak{W}^{n+1}$ in \eqref{errorRRd1d} and $v=\mathfrak{U}^{n+1}$ in \eqref{errorRRd2d} and use \eqref{errorRRd3d}, we have (cf. \cite[Lemma 3.3]{burman2024estimates}),

\begin{lemma}\label{lem:improveddsum}
It holds, for $n\ge3$,
\begin{alignat*}{1}
Z^{n+1}(\mathfrak{U},\mathfrak{W},\mathfrak{L})+S^{n+1}(\mathfrak{U},\mathfrak{W},\mathfrak{L})=Z^n(\mathfrak{U},\mathfrak{W},\mathfrak{L}) + \dt \mathfrak{F}^{n+1}  +\frac{\dt}{\alpha} \bl \mathcal{G}_2^{n+1}, \mathfrak{L}^{n+1} \br,
\end{alignat*}
where
\begin{alignat*}{1}
\mathfrak{F}^{n+1}:= & -(\mathcal{H}_1^{n+1}, \mathfrak{W}^{n+1})_s-(\mathcal{H}_2^{n+1}, \mathfrak{U}^{n+1})_f + \alpha \bl \mathcal{G}_1^{n+1}, \mathfrak{W}^{n+1} \br + \bl \mathfrak{U}^{n+1}-\mathfrak{U}^n, \mathcal{G}_2^{n+1} \br. 
\end{alignat*}
\end{lemma}
Hence, by taking the sum from $n=3$ to $n=N-1$,
\begin{lemma}\label{lemma:zdsum2d}
We have,
\begin{alignat}{1}\label{eq:zdsum2d}
Z^N(\mathfrak{U},\mathfrak{W},\mathfrak{L})+ \sum_{n=3}^{N-1} S^{n+1}(\mathfrak{U},\mathfrak{W},\mathfrak{L}) =  Z^3(\mathfrak{U},\mathfrak{W},\mathfrak{L}) + \dt \sum_{n=3}^{N-1} \mathfrak{F}^{n+1} + \frac{\dt}{\alpha} \sum_{n=3}^{N-1} \bl \mathcal{G}_2^{n+1}, \mathfrak{L}^{n+1} \br.
\end{alignat}  
\end{lemma}

Using Lemma \ref{lemma:zdsum2d}, we obtain

\begin{theorem}\label{thm:main}
  We have,
\begin{equation}\label{eq:Z2esti1}
\begin{aligned}
  Z^N&(\mathfrak{U},\mathfrak{W},\mathfrak{L})+ \sum_{n=3}^{N-1} S^{n+1}(\mathfrak{U},\mathfrak{W},\mathfrak{L}) +\frac12\|\mathfrak{W}^3\|^2_{L^2(\Omega_s)}+\frac12\|\mathfrak{U}^3\|^2_{L^2(\Omega_f)}\\
  &+\dt\nu_s\|\nabla\mathfrak{W}^3\|^2_{L^2(\Omega_s)}+\dt\nu_f\|\nabla\mathfrak{U}^3\|^2_{L^2(\Omega_f)}+\frac{\dt}{2\alpha}\|\mathfrak{L}^3\|^2_{L^2(\Sigma)}+\frac{\alpha\dt}{2}\|\mathfrak{U}^3\|^2_{L^2(\Sigma)}\le C(\dt)^2.
\end{aligned}
\end{equation}
\end{theorem}

\begin{proof}[Sketch of the proof]
Take $z=\mathfrak{W}^{3}$ in \eqref{eq:ddd11} and $v=\mathfrak{U}^{3}$ in \eqref{eq:ddd12} and use \eqref{eq:ddd13}, we have

\begin{equation}\label{eq:z2d}
\begin{aligned}
  \|&\mathfrak{W}^3\|^2_{L^2(\Omega_s)}+\|\mathfrak{U}^3\|^2_{L^2(\Omega_f)}+\dt\nu_s\|\nabla\mathfrak{W}^3\|^2_{L^2(\Omega_s)}+\dt\nu_f\|\nabla\mathfrak{U}^3\|^2_{L^2(\Omega_f)}\\
  &+\frac{\dt}{\alpha}\|\mathfrak{L}^3\|^2_{L^2(\Sigma)}+\alpha\dt\|\mathfrak{U}^3\|^2_{L^2(\Sigma)}\\
=&\dt\Big[\cR^3_1(\mathfrak{W}^3)-\cR_3^3(\mathfrak{W}^3)+\cR_2^3(\mathfrak{U}^3)+\bl\cG_2^3,\mathfrak{W}^3\br+\frac{1}{\alpha}\bl\cG_2^3,\mathfrak{L}^3\br\Big]:=\dt\mathcal{J}^3
\end{aligned}
\end{equation}

Adding \eqref{eq:z2d} and \eqref{eq:zdsum2d}, we obtain
\begin{equation}\label{eq:Z2estid}
\begin{aligned}
  Z^N&(\mathfrak{U},\mathfrak{W},\mathfrak{L})+ \sum_{n=3}^{N-1} S^{n+1}(\mathfrak{U},\mathfrak{W},\mathfrak{L}) +\frac12\|\mathfrak{W}^3\|^2_{L^2(\Omega_s)}+\frac12\|\mathfrak{U}^3\|^2_{L^2(\Omega_f)}\\
  &+\dt\nu_s\|\nabla\mathfrak{W}^3\|^2_{L^2(\Omega_s)}+\dt\nu_f\|\nabla\mathfrak{U}^3\|^2_{L^2(\Omega_f)}+\frac{\dt}{2\alpha}\|\mathfrak{L}^3\|^2_{L^2(\Sigma)}+\frac{\alpha\dt}{2}\|\mathfrak{U}^3\|^2_{L^2(\Sigma)}\\
  =&\dt\mathfrak{J}^3+\dt \sum_{n=3}^{N-1} \mathfrak{F}^{n+1} + \frac{\dt}{\alpha} \sum_{n=3}^{N-1} \bl \mathcal{G}_2^{n+1}, \mathfrak{L}^{n+1} \br.
\end{aligned}
\end{equation}
The terms involving $\mathfrak{J}^3$, $\mathcal{F}^{n+1}$ and the last term in \eqref{eq:Z2estid} can be estimated similarly as in \cite{burman2024estimates} and then we obtain the results immediately.
\end{proof}

\begin{remark}
Theorem \ref{thm:main} indicates that the term $\big(\dt\sum_{n=2}^{N-1} \nu_f\|\nabla(U^{n+1}-2U^n+U^{n-1})\|^2_{L^2(\Omega_f)}\big)^\frac12$ is of order $O((\dt)^3)$, which is a crucial ingredient for a defect correction scheme in \cite{ppcorrection}. The idea of the modified scheme is to construct the first time step \eqref{eq:firststep} such that the first step of the error equations of the second-order time difference \eqref{eq:ddd1} have this particular form, so that when performing the analysis, the $Z^3$ term on the right-hand side of \eqref{eq:zdsum2d} is canceled. We also observe $O((\dt)^2)$ convergence rates for the first difference of the errors and we also believe this idea can be used to prove the $H^2$ estimate in \cite{burman2024estimates}. We omit the details here.
\end{remark}

\section{Numerical Results}

In this section, we present three numerical examples to support our findings. We begin by recomputing Example 5.1 from \cite{burman2024estimates} to demonstrate that the improved scheme indeed achieves better convergence rates for the terms in $\Sigma_{n=2} S^{n+1}$. We then include two additional examples with different exact solutions. We define the following notations:
\begin{alignat*}{2}
    &e_u= \|U^N\|_{L^2(\Omega_f)},\quad &&e_{du}=  \|U^N-U^{N-1}\|_{L^2(\Omega_f)},\\
    &e_{dw}=  \|W^N-W^{N-1}\|_{L^2(\Omega_s)},\quad &&e_{gdu}=  \|\nabla(U^N-U^{N-1})\|_{L^2(\Omega_f)}.
  \end{alignat*}
and 
\begin{alignat*}{2}
  &e_{gdus}= (\Delta t\sum_{n=1}^{N-1}\|\nabla(U^{n+1}-U^n)\|^2_{L^2(\Omega_f)})^\frac12, &&e_{gdws}= (\Delta t\sum_{n=1}^{N-1}\|\nabla(W^{n+1}-W^n)\|^2_{L^2(\Omega_s)})^\frac12,\\
  &e_{dls}= (\Delta t\sum_{n=1}^{N-1}\|\Lambda^{n+1}-\Lambda^n\|^2_{L^2(\Sigma)})^\frac12, &&e_{gdu2s}= (\Delta t\sum_{n=2}^{N-1}\|\nabla(U^{n+1}-2U^{n}+U^{n-1})\|^2_{L^2(\Omega_f)})^\frac12,\\
  &e_{ggdus}= (\Delta t\sum_{n=1}^{N-1}\|D^2(U^{n+1}-U^n)\|^2_{L^2(\Omega_f)})^\frac12.\\
\end{alignat*}

\begin{example}\label{ex:ex1}
We take 
$$\uu=\ww=e^{-2\pi^2t}\cos(\pi x_1)\sin(\pi x_2),$$ $h=\Delta t$ and $\alpha=4$ in this example as in \cite[Example 5.1]{burman2024estimates}. We set $\Omega_f=[0,1]\times[0,0.75]$ and $\Omega_s=[0,1]\times[0.75,1]$. Here at level $k$ we set $\Delta t=(\frac12)^{k+1}$. 
\end{example}

We report the convergence rates in Tables \ref{table:ex1}-\ref{table:ex1s}. We observe the same convergence rates for quantities at final time $N$ as in \cite[Example 5.1]{burman2024estimates}. In contrast to \cite{erratum}, we observe $O((\dt)^2)$ for the sum of the first difference of the errors and $O((\dt)^3)$ for the sum of the second difference of the errors. These convergence rates coincide with our assumptions in \cite{ppcorrection}.

\begin{table}[H]
\centering
\caption{Convergence results for Example \ref{ex:ex3} at final time with linear finite element ($h=\dt$).}\label{table:ex1}
\begin{tabular}{ccccccccc}
\hline\\
\vspace{-0.2cm}
$k$&$e_{u}$&Order&$e_{du}$&Order&$e_{dw}$&Order&$e_{gdu}$&Order\\
\\
\hline
1&4.65e-01&-&2.05e-02&-&1.76e-02&-&8.18e-02&-\\
2&4.34e-01&0.10&6.09e-02&-1.57&3.67e-02&-1.06&2.33e-01&-1.51\\
3&4.25e-02&3.35&3.77e-02&0.69&2.35e-02&0.64&1.45e-01&0.69\\
4&1.50e-02&1.50&7.59e-03&2.31&4.02e-03&2.55&2.95e-02&2.30\\
5&5.77e-03&1.38&1.52e-03&2.32&7.17e-04&2.49&5.97e-03&2.30\\
6&2.46e-03&1.23&3.24e-04&2.23&1.43e-04&2.33&1.28e-03&2.22\\
7&1.12e-03&1.14&7.35e-05&2.14&3.12e-05&2.19&2.93e-04&2.14\\
8&5.26e-04&1.09&1.72e-05&2.09&7.20e-06&2.12&6.87e-05&2.09\\
\hline
\end{tabular}
\end{table}

\begin{table}[H]
\centering
\caption{Convergence results for Example \ref{ex:ex3} of the sum with linear finite element ($h=\dt$).}\label{table:ex1s}
\begin{tabular}{ccccccc}
\hline\\
\vspace{-0.2cm}
$k$&$e_{gdus}$&Order&$e_{gdws}$&Order&$e_{gdu2s}$&Order\\
\\
\hline
3&7.37e-02&-&9.42e-02&-&4.58e-02&-\\
4&4.07e-02&0.86&5.20e-02&0.86&1.18e-02&1.96\\
5&1.53e-02&1.41&1.92e-02&1.44&2.09e-03&2.50\\
6&4.64e-03&1.72&5.77e-03&1.73&3.08e-04&2.76\\
7&1.25e-03&1.89&1.56e-03&1.89&4.09e-05&2.91\\
8&3.08e-04&2.02&3.96e-04&1.98&5.03e-06&3.02\\
\hline
\end{tabular}
\end{table}

\begin{example}\label{ex:ex2}
We take 
$$\uu=\ww=(t^3+1)\cos(\pi x_1)\sin(\pi x_2),$$ in this example. Other parameters are identical as those in Example \ref{ex:ex1}.
\end{example}

We report the convergence rates in Tables \ref{table:ex2}–\ref{table:ex2s}. In this example, we use quadratic finite elements to eliminate any impact of the spatial discretization on the convergence. Similar convergence rates are observed as in Example \ref{ex:ex1}.

\begin{table}[H]
\centering
\caption{Convergence results for Example \ref{ex:ex2} at final time with quadratic finite element ($h=\dt$).}\label{table:ex2}
\begin{tabular}{ccccccccc}
\hline\\
\vspace{-0.2cm}
$k$&$e_{u}$&Order&$e_{du}$&Order&$e_{dw}$&Order&$e_{gdu}$&Order\\
\\
\hline
1&2.10e-01&-&1.45e-01&-&2.53e-02&-&5.87e-01&-\\
2&2.77e-02&2.93&1.88e-02&2.95&3.45e-03&2.87&7.59e-02&2.95\\
3&8.34e-04&5.05&2.87e-04&6.03&3.73e-04&3.21&1.15e-03&6.04\\
4&4.46e-04&0.90&7.33e-05&1.97&1.14e-04&1.71&2.93e-04&1.98\\
5&2.31e-04&0.95&1.83e-05&2.00&3.09e-05&1.88&7.32e-05&2.00\\
6&1.18e-04&0.98&4.55e-06&2.01&8.06e-06&1.94&1.83e-05&2.00\\
7&5.93e-05&0.99&1.14e-06&2.00&2.05e-06&1.97&4.56e-06&2.00\\
8&2.98e-05&0.99&2.83e-07&2.00&5.19e-07&1.99&1.14e-06&2.00\\
\hline
\end{tabular}
\end{table}

\begin{table}[H]
\centering
\caption{Convergence results for Example \ref{ex:ex2} of the sum with quadratic finite element ($h=\dt$).}\label{table:ex2s}
\begin{tabular}{ccccccc}
\hline\\
\vspace{-0.2cm}
$k$&$e_{gdus}$&Order&$e_{gdws}$&Order&$e_{gdu2s}$&Order\\
\\
\hline
3&4.07e-04&-&8.89e-04&-&4.81e-06&-\\
4&1.26e-04&1.69&2.58e-04&1.78&1.65e-06&1.55\\
5&3.43e-05&1.88&6.86e-05&1.91&3.77e-07&2.13\\
6&8.90e-06&1.95&1.76e-05&1.96&6.70e-08&2.49\\
7&2.26e-06&1.98&4.47e-06&1.98&1.06e-08&2.66\\
8&5.70e-07&1.99&1.13e-06&1.99&1.55e-09&2.78\\
\hline
\end{tabular}
\end{table}

\begin{example}\label{ex:ex3}
We take 
$$\uu=\ww=e^t\cos(\pi x_1)\sin(\pi x_2),$$ in this example. Other parameters are identical as those in Example \ref{ex:ex1}.
\end{example}

We report the convergence rates in Tables \ref{table:ex3}–\ref{table:ex3s}. In this example, we use quadratic finite elements to eliminate any impact of the spatial discretization on the convergence. Similar convergence rates are observed as in Example \ref{ex:ex1}.

\begin{table}[H]
\centering
\caption{Convergence results for Example \ref{ex:ex3} at final time with quadratic finite element ($h=\dt$).}\label{table:ex3}
\begin{tabular}{ccccccccc}
\hline\\
\vspace{-0.2cm}
$k$&$e_{u}$&Order&$e_{du}$&Order&$e_{dw}$&Order&$e_{gdu}$&Order\\
\\
\hline
1&5.34e-01&-&2.23e-01&-&5.90e-02&-&9.03e-01&-\\
2&2.18e-01&1.29&8.17e-02&1.45&2.36e-02&1.32&3.30e-01&1.45\\
3&6.58e-04&8.37&3.83e-05&11.06&5.10e-04&5.53&2.05e-04&10.65\\
4&3.29e-04&1.00&1.00e-05&1.93&1.31e-04&1.96&5.08e-05&2.01\\
5&1.65e-04&1.00&2.55e-06&1.98&3.33e-05&1.98&1.29e-05&1.98\\
6&8.24e-05&1.00&6.40e-07&1.99&8.40e-06&1.99&3.24e-06&1.99\\
7&4.12e-05&1.00&1.60e-07&2.00&2.11e-06&1.99&8.13e-07&2.00\\
8&2.06e-05&1.00&4.01e-08&2.00&5.28e-07&2.00&2.04e-07&2.00\\
\hline
\end{tabular}
\end{table}

\begin{table}[H]
\centering
\caption{Convergence results for Example \ref{ex:ex3} of the sum with quadratic finite element ($h=\dt$).}\label{table:ex3s}
\begin{tabular}{ccccccc}
\hline\\
\vspace{-0.2cm}
$k$&$e_{gdus}$&Order&$e_{gdws}$&Order&$e_{gdu2s}$&Order\\
\\
\hline
3&7.13e-05&-&1.39e-03&-&3.14e-06&-\\
4&2.05e-05&1.80&4.19e-04&1.73&6.45e-07&2.28\\
5&5.45e-06&1.91&1.12e-04&1.90&1.04e-07&2.63\\
6&1.40e-06&1.96&2.90e-05&1.96&1.56e-08&2.74\\
7&3.54e-07&1.98&7.35e-06&1.98&2.25e-09&2.79\\
8&8.91e-08&1.99&1.85e-06&1.99&3.13e-10&2.85\\
\hline
\end{tabular}
\end{table}

\bibliographystyle{abbrv}
\bibliography{referencesWP}

\end{document}